\documentclass[11pt, oneside]{article}   	
\usepackage{geometry}                		
\usepackage{graphicx}	
\usepackage{amssymb}
\usepackage{float}
\usepackage[cmex10]{amsmath}
\usepackage{mathrsfs,amsthm}
\usepackage{color}

\newcommand\bx{\boldsymbol x}
\newcommand\by{\boldsymbol y}
\newcommand\bE{\boldsymbol E}

\newcommand\bH{\boldsymbol H}
\newcommand\bF{\boldsymbol F}

\newcommand\bJ{\boldsymbol J}

\newcommand\bM{\boldsymbol M}
\newcommand\In{\operatorname{inc}}
\newcommand\Sc{\operatorname{scat}}
\newcommand\bn{\boldsymbol n}

\newtheorem{remark}{Remark}
\newtheorem{definition}{Definition}
\newtheorem{thm}{Theorem}
\newtheorem{lemma}{Lemma}

\newcommand{\Vo}{\mathcal V_0}
\newcommand{\To}{\mathcal T_0}
\newcommand{\Vw}{\mathcal V_\omega}
\newcommand{\Tw}{\mathcal T_\omega}

\title{Integral equation methods for electrostatics, acoustics
and electromagnetics in smoothly varying, anisotropic 
media\footnote{This work was supported in part by
the Spanish Ministry of Science and Innovation under project TEC2016-78028-C3-3-P
and the U.S. Department of Energy under grant DE-FG02-86ER53223.
}
}
\author{
Lise-Marie Imbert-Gerard\footnote{Department of Mathematics, University of Maryland. 
Email: {\tt lmig@math.umd.edu}}, 
Felipe Vico\footnote{Instituto de Telecomunicaciones y Aplicaciones 
Multimedia (ITEAM), Universidad Polit\` ecnica
de Val\` encia, 46022 Val\` encia, Spain. 
{{\em email}:\ {\tt {felipe.vico@gmail.com}}},
{\ {\tt {mferrand@dcom.upv.es}}}
}, 
Leslie Greengard\footnote{Courant Institute, New York University,
  and Flatiron Institute, Simons Foundation, New York, NY. Email:
  {\tt greengard@cims.nyu.edu}}, and
Miguel Ferrando\footnotemark[3]
}

\begin{document}
\maketitle
\begin{abstract}
We present a collection of well-conditioned integral equation methods
for the solution of electrostatic, acoustic or electromagnetic scattering 
problems involving anisotropic, inhomogeneous media.
In the electromagnetic case,
our approach involves a minor modification of a classical formulation. 
In the electrostatic or acoustic setting, 
we introduce a new vector partial differential
equation, from which the desired solution is easily obtained. 
It is the vector equation for which we derive a well-conditioned integral 
equation. In addition to providing a unified framework for these
solvers, we illustrate their performance using iterative solution methods
coupled with the FFT-based technique of \cite{Volintfft} to discretize and 
apply the relevant integral operators.
\end{abstract}

%
%
%

\section{Introduction}

In this paper, we develop fast, high order accurate integral equation methods for
several classes of elliptic partial differential equations (PDEs) in three dimensions
involving anisotropic, inhomogeneous media.
In the electrostatic setting, we consider the 
anisotropic Laplace equation
\begin{equation}\label{Laplace0}
\begin{aligned}
\nabla\cdot\epsilon(\bx)\nabla\phi(\bx)=0,\ \ \bx\in\mathbb{R}^3 \, ,
\end{aligned}
\end{equation}
where $\epsilon(\bx)$ is a real, $3 \times 3$ symmetric matrix, subject to 
certain regularity conditions discussed below. We also assume
$\epsilon(\bx)$ is a compact perturbation of the identity operator $I$ - that is,
$\epsilon(\bx) - I$ has compact support. 
A typical objective is to determine the response of the inclusion to a known, 
applied static field, with the response satisfying 
suitable decay conditions at infinity.

For acoustic or electromagnetic modeling in the frequency domain, 
we consider the anisotropic Helmholtz equation
\begin{equation}\label{Helmholtz0}
\begin{aligned}
\nabla\cdot\epsilon^{-1}(\bx)\nabla\phi(\bx)+\omega^2\phi(\bx)=0,\ \ \bx\in\mathbb{R}^3
\end{aligned}
\end{equation}
and the anisotropic Maxwell equations
\begin{equation}\label{Maxwell0}
\begin{aligned}
\nabla\times\bE(\bx)&=i\omega\mu(\bx)\bH(\bx)\\
\nabla\times\bH(\bx)&=-i\omega\epsilon(\bx)\bE(\bx),\ \ \ \ \ \ \ \ \bx\in\mathbb{R}^3,
\end{aligned}
\end{equation}
respectively.
Here, $\epsilon(\bx)$ and $\mu(\bx)$ are complex-valued $3 \times 3$ matrices, 
subject to regularity and spectral properties to be discussed in detail. 
We again assume that $\epsilon(\bx)$ and $\mu(\bx)$ are compact perturbation of the 
identity.
A typical objective is to determine the response of the inclusion to an impinging
acoustic or electromagnetic wave, with the response satisfying 
suitable radiation conditions at infinity.
For a thorough discussion of the origins and applications of these problems, we
refer the reader to the textbooks \cite{CHEW,JACKSON,LANDAU}.

Instead of discretizing the partial differential equations (PDEs) themselves,
we will develop integral representations of the solution that satisfy
the outgoing decay/radiation conditions exactly, avoiding the need for
truncating the computational domain and imposing approximate outgoing boundary 
conditions.
The resulting integral equations will be shown to involve equations
governed by operators of the form $I+B+K$ where $I$ is the identity, 
$B$ is a linear contraction mapping, $K$ is compact. 
A simple argument based on the Neumann series will allow us to extend the 
Fredholm alternative to this setting (and therefore to prove existence for 
the original, anisotropic, elliptic PDEs themselves). 
Moreover, our formulations permit
high-order accurate discretization, and FFT-based acceleration on uniform grids. 
In the electromagnetic case,
our approach is closely related to some classical formulations. 
For the Laplace and Helmholtz equations, however, our approach appears to 
be new and depends on the construction of a {\em vector} PDE
from which the desired solution is easily obtained. It is the vector PDE
for which we will derive a new, well-conditioned integral 
equation. 
One purpose of the present paper is to describe all of these solvers in a 
unified framework. Given that resonance-free second kind integral equations 
are typically well-conditioned, they are suitable for solution using simple iterative
methods such as GMRES and BICGSTAB without any preconditioner. 
We use the method of \cite{Volintfft} to discretize and apply the integral operators
with high order accuracy and demonstrate the performance of our scheme with several
numerical examples.
 
We will use the language of scattering theory throughout. Thus, 
for the scalar equations, we write 
$\phi=\phi^{\In}+\phi^{\Sc}$ where $\phi^{\In}$ is a known function that satisfies 
the homogeneous, isotropic Laplace or Helmholtz equation in free space away from 
sources. In the electrostatic case, 
$\phi^{\Sc}$ is assumed to satisfy the decay condition
\begin{equation}
\label{Laplace_Scalar_RC}
\lim_{r \rightarrow \infty} \phi^{\Sc} = o(1),
\end{equation}
where $r = |\bx|$.
In the acoustic case \cite{CK1},
$\phi^{\Sc}$ is assumed to satisfy the Sommerfeld radiation condition
\begin{equation}
\label{helmrad}
\lim_{r \rightarrow \infty} r \left(\frac{\partial \phi^{\Sc}}{\partial r} 
- i\omega \phi^{\Sc} \right) = 0.
\end{equation}
In the electromagnetic setting
(\ref{Maxwell0}), we write $\bE=\bE^{\In}+\bE^{\Sc}$ and 
$\bH=\bH^{\In}+\bH^{\Sc}$, with $\bE^{\In},\bH^{\In}$ corresponding to a 
known solution of the homogeneous, isotropic Maxwell equations 
in free space away from sources.
$\bE^{\Sc},\bH^{\Sc}$ are assumed to satisfy the Silver-M\"uller radiation condition 
\cite{CK1}: 
\begin{equation}
\label{silverrad}
\lim_{r \rightarrow \infty} \left(\bH^{\Sc} \times {\bx} - r \bE^{\Sc} \right) = 0.
\end{equation}

\section{The anisotropic Laplace equation}

We first consider the electrostatic problem
\eqref{Laplace0},
where $\epsilon(\bx)$ is a real, symmetric, positive definite $3\times3$ matrix 
with bounded entries, such that $\epsilon(\bx)-I$ has compact 
support. We assume that 
the smallest eigenvalue  of $\epsilon(\bx)$ is bounded away from zero,
so that the PDE is uniformly elliptic (see \cite{evansPDE}, p.294).

Letting $\phi=\phi^{\In}+\phi^{\Sc}$, we assume that
$\phi^{\In}(\bx)$ is a given ``applied field" with
$\Delta \phi^{\In}(\bx) = 0$ in the support of $\epsilon(\bx)-I$. 

\begin{definition} By the {\em anisotropic electrostatic scattering
problem} (or simply the {\em electrostatic scattering problem}) we mean the 
calculation of a function 
$\phi^{\Sc}(\bx)\in H^1_{loc}(\mathbb{R}^3)$ that satisfies the 
radiation condition \eqref{Laplace_Scalar_RC} and the equation
\begin{equation}\label{Laplace_Scalar}
\begin{aligned}
\nabla\cdot\epsilon(\bx)\nabla\phi^{\Sc}(\bx)=-\nabla\cdot(\epsilon(\bx)-I)\nabla 
\phi^{\In}(\bx)
\end{aligned}
\end{equation}
for $\bx\in\mathbb{R}^3$.
Given this function $\phi^{\Sc}$, $\phi=\phi^{\In}+\phi^{\Sc}$ clearly
satisfies the eq. \eqref{Laplace0}.
\end{definition}

Because of the mild regularity assumptions made on $\epsilon(\bx)$ and on the 
solution $\phi^{\Sc}$, we must interpret eq. \eqref{Laplace_Scalar} in a weak sense (eq. \eqref{Laplace_Scalar_Weak} below).

We begin by proving a uniqueness result.

\begin{thm}\label{Unique_Laplace_Scalar}
The anisotropic electrostatic scattering problem has at most one solution.
\end{thm}

\begin{proof}
Let $B_R$ be an open ball centered at the origin that contains 
the support of $\epsilon(\bx)-I$, and let $T$ denote the 
Dirichlet-to-Neumann (DtN) map for harmonic functions in the exterior of $B_R$. 
Integrating \eqref{Laplace_Scalar} by parts, we obtain the weak formulation
on $B_R$:
\begin{equation} \label{Laplace_Scalar_Weak}
\int_{B_R}\nabla\psi(\bx)\epsilon(\bx)\nabla\phi^{\Sc}(\bx) dV =\int_{\partial B_R}\psi T[\phi^{\Sc}] dS-\int_{B_R}\nabla\psi(\bx)(\epsilon-I)(\bx)\nabla\phi^{\In}(\bx) dV
\end{equation}
for $\psi\in H^1(B_R)$.
Assuming no ``incoming" data ($\phi^{\In} =0$) and using
$\psi=\phi^{\Sc}$ itself as a testing function yields
\begin{equation}\label{eq_estimate}
\int_{B_R}\nabla\phi^{\Sc}(\bx)\epsilon(\bx)\nabla\phi^{\Sc}(\bx) dV =\int_{\partial B_R}\phi^{\Sc} T[\phi^{\Sc}] dS.
\end{equation}

From the uniform ellipticity of $\epsilon(\bx)$, for some constant $C>0$ we have
\begin{equation}\label{eq_estimate2}
C\int_{B_R}|\nabla\phi^{\Sc}(\bx)|^2dV \le\int_{\partial B_R}\phi^{\Sc}T[\phi^{\Sc}] dS.
\end{equation}

On the surface of $B_R$, the 
DtN map can be computed in the spherical harmonics basis, with 
\[ T[Y_{nm}](\theta,\phi)=-(n+1)Y_{nm}(\theta,\phi). \]
From this, it is straightforward to see that
\begin{equation}
\int_{B_R}|\nabla\phi^{\Sc}(\bx)|^2dV=0 \Rightarrow \phi^{\Sc}(\bx)=c\ 
 \forall \bx \in B_R.
\end{equation}
Using (\ref{eq_estimate2}), we have
\begin{equation}
0\le\int_{\partial B_R}\phi^{\Sc} T[\phi^{\Sc}] dS=\int_{\partial B_R}c T[c] dS
=-|c|^2 A,
\end{equation}
where $A$ is the surface area of $B_R$.
Thus, $c=0$, yielding the desired result.
\end{proof}

\begin{remark}
While the standard proof of existence for the anisotropic Laplace equation
relies on the Lax-Milgram theorem \cite{evansPDE}, we turn now to an alternate
approach, based on deriving an auxiliary vector PDE and a corresponding, 
well-conditioned, Fredholm integral equation. 
\end{remark}

Note first that
eq. \eqref{Laplace_Scalar} can be rewritten in the form
\begin{equation}
\nabla\cdot\epsilon\nabla\phi^{\Sc}=\Delta \phi^{\Sc}+\nabla\cdot(\epsilon-I)\nabla\phi^{\Sc}=-\nabla\cdot\epsilon\nabla\phi^{\In}=-\nabla\cdot(\epsilon-I)\nabla\phi^{\In},
\end{equation}
since $\phi^{\In}$ is harmonic in $B_R$, a region which contains
the support of  $\epsilon(\bx)-I$.
Recall that $\phi^{\Sc}$ satisfies the Laplace decay condition 
\eqref{Laplace_Scalar_RC}. 

Suppose now that
$\bF$ is a vector field such that $\nabla\cdot\bF=\phi^{\Sc}$
and satisfies \eqref{Laplace_Scalar_RC}.
Note, however, that $\nabla \times \bF$ is, as yet, unspecified. 
Then, $\bF$ satisfies the following equation:
\begin{equation}
\Delta \nabla\cdot\bF+\nabla\cdot(\epsilon-I)\nabla\nabla\cdot\bF=-\nabla\cdot(\epsilon-I)\nabla\phi^{\In},
\end{equation}
or
\begin{equation}\label{div_aug}
\nabla\cdot\Delta \bF+\nabla\cdot(\epsilon-I)\nabla\nabla\cdot\bF=-\nabla\cdot(\epsilon-I)\nabla\phi^{\In}.
\end{equation}	
We now define the auxiliary equation for $\bF$ (which will determine its curl):
\begin{equation}\label{aug_anisotrop}
\Delta \bF+(\epsilon-I)\nabla\nabla\cdot\bF=-(\epsilon-I)\nabla\phi^{\In} \, .
\end{equation}
Since eq. (\ref{div_aug}) is obtained by taking the divergence of 
(\ref{aug_anisotrop}), it is natural to introduce the following definition.

\begin{definition} By the {\em vector electrostatic scattering problem}
we mean the calculation 
of a vector function 
$\mathbf{F}^{\Sc}(\bx)\in H^2_{loc}(\mathbb{R}^3)$ that satisfies
\begin{equation}\label{Laplace_Vector}
\begin{aligned}
\Delta \mathbf{F}^{\Sc}+(\epsilon-I)\nabla\nabla\cdot\mathbf{F}^{\Sc}=-(\epsilon-I)\nabla\nabla\cdot\mathbf{F}^{\In} \\
\end{aligned}
\end{equation}
and the standard decay condition at infinity:
\begin{equation}\label{Laplace_Vector_RC}
\begin{aligned}
\mathbf{F}^{\Sc}(\bx)=o(1), \ \ {\rm as}\ |\bx|\rightarrow \infty,
\end{aligned}
\end{equation}
uniformly in all directions. Here, $\mathbf{F}^{\In}(\bx)$ is an $H^2$ function 
defined on an open set that contains the support of $\epsilon(\bx)-I$,
where it satisfies $\Delta \mathbf{F}^{\In}(\bx) = 0$.
\end{definition}

Due to the regularity properties imposed on $\mathbf{F}^{\Sc}$ and the lack of 
derivatives acting on the entries of $\epsilon_{ij}(\bx)$, 
eq. (\ref{Laplace_Vector_RC}) does not need to be interpreted in a weak sense.
We now establish a relation between the vector and scalar problems.

\begin{lemma}\label{Lemma_Laplace_div}
If $\mathbf{F}^{\Sc}(\bx)\in H^2_{loc}(\mathbb{R}^3)$ satisfied the 
vector electrostatic scattering problem, then 
$\phi^{\Sc}(\bx):=\nabla\cdot\mathbf{F}^{\Sc}(\bx)\in H^1_{loc}(\mathbb{R}^3)$ and
satisfies the anisotropic electrostatic scattering problem with 
right hand side given by the 
incoming field $\phi^{\In}=\nabla\cdot \mathbf{F^{\In}}$.
\end{lemma}

\begin{proof}
If we assume that $\mathbf{F}^{\Sc}(\bx)\in H^3_{loc}(\mathbb{R}^3)$, 
then the result follows immediately by taking the divergence of 
eq. (\ref{Laplace_Vector}) and interchange the order of the operators.
In general, however, we cannot assume such regularity for $\mathbf{F}^{\Sc}(\bx)$, 
in particular if there are discontinuities in the entries $\epsilon(\bx)_{ij}$.

Thus, for the general case, we let
$B_R$ be an open ball centered at the origin that contains the support of 
$\epsilon(\bx)-I$. In the region $\mathbb{R}^3 \backslash B_R$, the governing 
differential operator is the isotropic Laplacian $\Delta$ 
and, by standard regularity results
(see, for example, Corollary 8.11 in \cite{elliptic_pde}), 
the solution $\mathbf{F}_i^{\Sc} \in C^{\infty}$. Thus, we can interpret the 
radiation condition in the strong sense. We can now apply the representation theorems 
4.11 and 4.13 from \cite{CK1} in the region $\mathbb{R}^3 \backslash B_R$, in the 
limit $k=0$. From this, it follows that
$\nabla\cdot \bF^{\Sc}=O\left(\frac{1}{|\bx|^2}\right)$ uniformly in all directions. 
Thus, $\phi^{\Sc}(\bx) = \nabla\cdot \bF^{\Sc}$ satisfies the 
radiation condition \eqref{Laplace_Scalar_RC} for the 
anisotropic electrostatic scattering problem.

Let $\psi\in H^{1}(B_R)$. From eq. (\ref{Laplace_Vector}), we can write
\begin{equation}
\begin{aligned}
\nabla\psi\cdot\Big(\Delta \mathbf{F}^{\Sc}+(\epsilon-I)\nabla\nabla\cdot\mathbf{F}^{\Sc}\Big)=-\nabla\psi(\bx)\cdot(\epsilon-I)\nabla\nabla\cdot\mathbf{F}^{\In}.
\end{aligned}
\end{equation}
Using the vector identity $\Delta \mathbf{F}^{\Sc}=-\nabla\times\nabla\times\mathbf{F}^{\Sc}+\nabla\nabla\cdot\mathbf{F}^{\Sc}$, we have
\begin{equation}
\begin{aligned}
-\nabla\psi\cdot\nabla\times\nabla\times\mathbf{F}^{\Sc}+\nabla\psi\cdot\epsilon\nabla\nabla\cdot\mathbf{F}^{\Sc}=-\nabla\psi\cdot(\epsilon-I)\nabla\nabla\cdot\mathbf{F}^{\In}.
\end{aligned}
\end{equation}
Integrating over the ball $B_R$, we obtain
\begin{equation}\label{prec}
\begin{aligned}
\int_{B_R}-\nabla\psi\cdot\nabla\times\nabla\times\mathbf{F}^{\Sc}dV+&\int_{B_R}\nabla\psi\cdot\epsilon\nabla\nabla\cdot\mathbf{F}^{\Sc}dV=\\
-&\int_{B_R}\nabla\psi\cdot(\epsilon-I)\nabla\nabla\cdot\mathbf{F}^{\In}dV.
\end{aligned}
\end{equation}

We will make use of the
following identity that holds for every function  $\mathbf{F}\in H^2(B_R)$, $\psi\in H^1(B_R)$ and is straightforward to derive from the divergence theorem:
\begin{equation}\label{IDN}
\begin{aligned}
\int_{B_R} \nabla\psi\cdot\nabla\times\nabla\times\mathbf{F}dV=\int_{\partial B_R}\psi\bn\cdot\nabla\times\nabla\times\mathbf{F}dS.
\end{aligned}
\end{equation}
Combining \eqref{IDN} and \eqref{prec}, we obtain
\begin{equation}
\begin{aligned}
\int_{B_R}\nabla\psi\cdot\epsilon\nabla\nabla\cdot\mathbf{F}^{\Sc}dV&=\int_{\partial B_R}\psi\bn\cdot\nabla\times\nabla\times\mathbf{F}^{\Sc}dV\\
-&\int_{B_R}\nabla\psi\cdot(\epsilon-I)\nabla\nabla\cdot\mathbf{F}^{\In}dV.
\end{aligned}
\end{equation}
Since $\Delta\mathbf{F}^{\Sc}(\bx)=0$ for $\bx\in \partial B_R$, we may write
\begin{equation}
\int_{\partial B_R}\psi\bn\cdot\nabla\times\nabla\times\mathbf{F}^{\Sc}dV
=\int_{\partial B_R}\psi\bn\cdot\nabla\nabla\cdot\mathbf{F}^{\Sc}dV.
\end{equation}
Thus, 
\begin{equation}
\begin{aligned}
\int_{B_R}\nabla\psi(\bx)\cdot\epsilon\nabla\nabla\cdot\mathbf{F}^{\Sc}dV=\int_{\partial B_R}\psi T[\nabla\cdot \mathbf{F}^{\Sc}]dS-\int_{B_R}\nabla\psi(\bx)\cdot(\epsilon-I)\nabla\nabla\cdot\mathbf{F}^{\In}dV.
\end{aligned}
\end{equation}

In short,
for $\phi^{\In}:=\nabla\cdot \mathbf{F}^{\In}$, 
$\phi^{\Sc}:=\nabla\cdot\mathbf{F}^{\Sc}$ satisfes \eqref{Laplace_Scalar_Weak}, the 
weak version of the anisotropic electrostatic scattering problem.
\end{proof}

\begin{thm}[Uniqueness]
The vector electrostatic scattering problem has at most one solution.
\end{thm}
\begin{proof}
Let $\bF^{\Sc}\in H^2_{loc}(\mathbb{R}^3)$ be a solution of the homogeneous equation 
\begin{equation}\label{aug_anisotrop_homog}
\Delta \bF^{\Sc}+(\epsilon-I)\nabla\nabla\cdot\bF^{\Sc}=0
\end{equation}
that satisfies the radiation condition.
By Lemma \ref{Lemma_Laplace_div}, $\nabla\cdot\mathbf{F}^{\Sc}$ satisfies 
the homogeneous weak formulation  (\ref{Laplace_Scalar_Weak}) and, 
as a consequence, using Theorem \ref{Unique_Laplace_Scalar}, 
$\nabla\cdot\mathbf{F}^{\Sc}=0$. From eq. \eqref{aug_anisotrop_homog},
it then follows that $\Delta\mathbf{F}^{\Sc}(\bx)=0$ for all $\bx \in \mathbb{R}^3$, 
so that $\mathbf{F}^{\Sc}=0$ for all $\bx \in \mathbb{R}^3$.
\end{proof}

We turn now to the question of existence, for which we 
define the volume integral operator
\begin{equation}\label{Vo}
\Vo(\bJ):=\int_{\mathbb{R}^3} \frac{1}{4\pi|\bx-\by|}\bJ(\by)dV_{\by}. 
\end{equation}
It is well-known (see, for example, Theorem 8.2 in \cite{CK2}) that
$\Vo:L^2(\Omega)\rightarrow H^2(\Omega)$ is a continuous map
for any bounded open set $\Omega\subset\mathbb{R}^3$ and that
\begin{equation}
\begin{aligned}
\Delta \Vo(\bJ)&=-\bJ. \\
\end{aligned}
\end{equation}
We define the operator $\To$ by
\begin{equation}\label{Todef}
\To(\bJ):=\frac{\bJ}{2}+\nabla\nabla\cdot \Vo(\bJ).
\end{equation}

\begin{lemma}\label{lemma1}
Let $H_{\epsilon}$ denote the operator mapping 
$L^2(\mathbb{R}^3) \rightarrow L^2(\mathbb{R}^3)$ with
\begin{equation} \label{rhodef}
H_{\epsilon}(\bx)\bJ(\bx) = (\epsilon(\bx)+I)^{-1}(\epsilon(\bx)-I) \bJ(\bx) .
\end{equation}
Then, 
$ \| H_{\epsilon}(\bx) \|_2 < 1$.
\end{lemma}
\begin{proof}
For $z \in \mathbb{C}$, let $f(z)=\frac{z-1}{z+1}$. $f$ maps
the open half space $\Re{z}>0$ to $|f(z)|<1$.
Since
$\epsilon(\bx)$ is assumed to be real symmetric and uniformly elliptic,
it is expressible in diagonal form as $\epsilon(\bx)=U(\bx)D(\bx)U^*(\bx)$, 
with the diagonal elements of $D(\bx)$ positive and bounded away from zero
\cite{evansPDE}. It follows that 
$\rho_{\epsilon}(\bx)=U(\bx)f(D)(\bx)U^*(\bx)\in L^{\infty}(\mathbb{R}^3)$,
with eigenvalues strictly bounded by one, proving the desired result.
\end{proof}

\begin{lemma}\label{lemma2}
The operator $2\To(\bJ)$ is an isometry on $L^2(\mathbb{R}^3)$.
That is, $\|2\To\|_{L^2(\mathbb{R}^3)}=1.$
\end{lemma}
\begin{proof}
Using the Helmholtz decomposition (see, for example, Theorem 14 in
\cite{Hodge_L2}), 
we can write $\bJ=\nabla\psi+\nabla\times\mathbf{P}$.
It is straightforward to check that
$2\To(\nabla\times\mathbf{P})=\nabla\times\mathbf{P}$, while 
 $2\To(\nabla\phi)=-\nabla\phi$. Thus, 
$2\To(\bJ)=-\nabla\psi+\nabla\times\mathbf{P}$ and the result follows 
immediately from the orthogonality of the Helmholtz decomposition.
\end{proof}

\begin{thm}
There exist solutions to the anisotropic scalar and vector electrostatic
scattering problems.
\end{thm}
\begin{proof}
We note first that the vector field $\bF:=\Vo(\bJ)$ satisfies eq. 
\eqref{aug_anisotrop} if and only if $\bJ$ satisfies
\begin{equation}
-\bJ+(\epsilon-I)\nabla\nabla\cdot \Vo(\bJ)=-(\epsilon-I)\nabla\phi^{\In}.
\end{equation}
From \eqref{Todef}, this is equivalent to
\begin{equation}
-\bJ+(\epsilon-I)\Big(-\frac{\bJ}{2}+\To(\bJ)\Big)=-(\epsilon-I)\nabla\phi^{\In}.
\end{equation}
Multiplying by $-2(\epsilon+I)^{-1}$ and a little algebra yields
\begin{equation}\label{eq:intLap}
( I+ B) \bJ = 2H_\epsilon\nabla \phi^{\In},
\end{equation}
where $H_\epsilon$ is defined in \eqref{rhodef} and
$B(\bJ):=-2H_{\epsilon}\To(\bJ)$.
Eq. (\ref{eq:intLap}) is an integral equation 
for $\bJ$.
Moreover, from Lemmas \ref{lemma1} and \ref{lemma2},
the operator
$B: L^2(B_R) \rightarrow L^2(B_R)$
satisfies $\|B\|<1$, so that $(I+B)$ is invertible using the Neumann series:
$$\bJ=2(I+B)^{-1}H_{\epsilon}\nabla\phi^{\In}
=2\sum_{n=0}^{\infty}B^nH_{\epsilon}\nabla\phi^{\In}.$$

Since the volume integral operator $\Vo$ maps
$L^2(B_R)$ to $H^2(B_R)$, it follows that 
$\mathbf{F}:=\Vo(\bJ)\in H^2_{loc}(\mathbb{R}^3)$. Thus, 
$\mathbf{F}$ is a solution to the vector electrostatic scattering problem, 
and $\nabla\cdot\mathbf{F}\in H^1_{loc}(\mathbb{R}^3)$ is a solution to the 
anisotropic electrostatic scattering problem.
\end{proof}
The same result holds in the two-dimensional case.

\begin{remark}[Smoothness of the coefficients]\label{Num_comments}
From a practical viewpoint, the integral equation \eqref{eq:intLap} can be
discretized using a Nystr\"om method and solved iteratively to obtain 
a numerical solution of the original scalar problem (\ref{Laplace_Scalar}). 
It is 
worth noting that no estimate involving derivatives of $\epsilon_{ij}(\bx)$ is required.
Because it is a Fredholm equation of the second kind, 
the order of accuracy obtained in the solution is the 
same as the order of accuracy used in the underlying quadrature rule
\cite{ANSELONEBOOK}.
Of course, if there are jumps in $\epsilon(\bx)$, then adaptive discretization methods 
are recommended for resolution, but additional unknowns and surface integral 
operators are not required to account for the effect of these discontinuities.
\end{remark}

\section{The anisotropic Helmholtz equation}

In this section, we assume that the matrix $\epsilon(\bx)$ has entries in 
$C^1(\mathbb{R}^3)$, although we will discuss some of the issues
raised in relaxing this assumption. We will also assume that 
$\epsilon(\bx)$  can be diagonalized in the form
$$\epsilon(\bx)=U(\bx)D(\bx)U^*(\bx),$$ 
where $U(\bx)$ is a unitary complex matrix and
$D(\bx)$ is a diagonal matrix with positive definite real part (with entries 
bounded away from zero) and a positive semi-definite imaginary part 
(see \cite{Potthast1}). 
We will also assume that $\epsilon(\bx)-I$ 
has compact support, where $I$ is the $3\times 3$ identity matrix.
After proving uniqueness for the anisotropic Helmholtz equation, we
introduce a related vector Helmholtz equation that will be used to establish 
existence using Fredholm theory.

\begin{definition}\label{Definition_IASH}
By the {\em anisotropic Helmholtz scattering problem},
 we mean the determination of a function 
$\phi^{\Sc}(\bx)\in H^1_{loc}(\mathbb{R}^3)$ that satisfies the equation:
\begin{equation}\label{Helmholtz_Scalar_strong}
\begin{aligned}
\nabla\cdot\epsilon^{-1}(\bx)\nabla\phi^{\Sc}(\bx)+\omega^2\phi^{\Sc}(\bx)=-\nabla\cdot(\epsilon^{-1}(\bx)-I)\nabla\phi^{\In},\ \ \bx\in\mathbb{R}^3,
\end{aligned}
\end{equation}
where $\phi^{\In}$ is a known function with
$$\Delta \phi^{\In}(\bx)+\omega^2\phi^{\In}(\bx)=0$$ 
in the support of $\epsilon(\bx)-I$. 
$\phi^{\Sc}(\bx)$ must also satisfy the Sommerfeld radiation condition
\eqref{helmrad}.
\end{definition}

\begin{thm}[Uniqueness]\label{Unique_Helmholtz_scalar}
The anisotropic Helmholtz scattering problem has at most one solution.
\end{thm}
\begin{proof}
The result follows from arguments analogous to those presented in 
section 2 of \cite{ucp_disc_1}.
Let $B_R$ be an open ball centered at the origin that covers the 
support of $\epsilon(\bx)-I$. Assuming the right-hand side is zero,
we can write \eqref{Helmholtz_Scalar_strong} in weak form by making
use of the Dirichlet to Neumann operator $T$ for the exterior of 
the sphere $S_R=\partial B_R$:
\begin{equation}
\label{Helmholtz_Scalar_weak}
\int_{\partial B_R}\psi T(\phi^{\Sc})dS =
\int_{B_R}\nabla \psi(\bx) \cdot \epsilon(\bx)^{-1}\nabla \phi^{\Sc}(\bx) 
-\psi(\bx)\omega^2 \phi^{\Sc}(\bx)dV,
\end{equation}
for all $\psi\in H^1(B_R)$.
Letting $\psi=\overline{\phi^{\Sc}}$ and taking complex conjugates, we have
\begin{equation}
\int_{\partial B_R}\phi^{\Sc}\frac{\partial \overline{\phi^{\Sc}}}{\partial n}dS
=\int_{B_R}\nabla \phi^{\Sc} \cdot 
(\bx)\overline{\epsilon(\bx)^{-1}}\overline{\nabla \phi^{\Sc}(\bx)} 
-\omega^2 |\phi^{\Sc}(\bx)|^2dV,
\end{equation}
so that
\begin{equation}
\Im \bigg(\int_{\partial B_R}\phi^{\Sc}\frac{\partial \overline{\phi^{\Sc}}}{\partial n}dS\bigg)=
\Im \bigg(\int_{B_R}\nabla \phi^{\Sc} \cdot (\bx)\overline{\epsilon(\bx)^{-1}}\overline{\nabla \phi^{\Sc}(\bx)} \bigg).
\end{equation}
Moreover from our assumptions about $\epsilon$, 
namely that $\epsilon(\bx)=U(\bx)D(\bx)U^*(\bx)$, 
the right-hand side can be written as
$$\nabla \phi^{\Sc} \cdot (\bx)\overline{\epsilon(\bx)^{-1}}\overline{\nabla \phi^{\Sc}(\bx)}=\xi(\bx) \cdot \overline{D^{-1}(\bx)\xi(\bx)}=\sum_{i=1}^3|\xi_i(\bx)|^2\overline{D^{-1}_{ii}(\bx)},$$
where $\xi(\bx) = \nabla U^*(\bx) \phi^{\Sc}$.
Thus,
\begin{equation}
\Im \bigg(\int_{\partial B_R}\phi^{\Sc}\frac{\partial \overline{\phi^{\Sc}}}{\partial n}dS\bigg)\ge 0.
\end{equation}
From Rellich's lemma \cite{CK1}, we may conclude that 
$\phi^{\Sc}(\bx)=0, \forall \bx\in\mathbb{R}^3/ B_R$. It then follows
from Theorem 1 in Section 6.3.1 of  \cite{evansPDE}  that 
$\phi^{\Sc}\in H^2(B_R)$. As a result, eq. \eqref{Helmholtz_Scalar_strong} is 
satisfied in a strong sense and we can use the unique continuation theorem 
(Theorem 17.2.6, \cite{hormander_III}) to conclude 
that $\phi^{\Sc}(\bx)=0$ for all $\bx\in\mathbb{R}^3$.
\end{proof}

As in the electrostatic case, the essential idea underlying the derivation of 
a well-conditioned formulation involves recasting the scalar problem 
of interest in terms of a vector-valued PDE.

\begin{definition} By the {\em vector Helmholtz scattering problem} we 
mean the determination of a vector function 
$\mathbf{F}^{\Sc}(\bx)\in H^2_{loc}(\mathbb{R}^3)$ satisfying the equation
\begin{equation}\label{Helmholtz_Vector}
\begin{aligned}
\Delta \mathbf{F}^{\Sc}+\omega^2\mathbf{F}^{\Sc}+(\epsilon^{-1}-I)\nabla\nabla\cdot\mathbf{F}^{\Sc}=-(\epsilon^{-1}-I)\nabla\nabla\cdot\mathbf{F}^{\In}, \\
\end{aligned}
\end{equation}
where $\mathbf{F}^{\In}(\bx)$ is a function defined on the support of 
$\epsilon(\bx)-I$ satisfying the homogeneous equation 
$\Delta\bF^{\In}+\omega^2\bF^{\In}=0$ and the standard radiation condition 
\begin{equation}\label{radAHm}
\begin{array}{ll}
\nabla\times \mathbf{F}^{\Sc}(\bx)\times\frac{\bx}{|\bx|}+\frac{\bx}{|\bx|}\nabla\cdot \mathbf{F}^{\Sc}(\bx)-ik\mathbf{F}^{\Sc}(\bx)=o\Big(\frac{1}{|\bx|}\Big), \quad &|\bx|\rightarrow \infty\, ,
\end{array}
\end{equation}
uniformly in all directions $\frac{\bx}{|\bx|}$.
\end{definition}

Note that in the vector Helmholtz scattering problem,
the entries of $\epsilon^{-1}$ are not acted on by a differential operator.
Thus, we will consider solutions of 
(\ref{Helmholtz_Vector}) in a strong sense, without loss of generality.

\begin{lemma}\label{Lemma1_div}
If $\mathbf{F}^{\Sc}(\bx)\in H^2_{loc}(\mathbb{R}^3)$ satisfies the 
vector Helmholtz scattering problem in a strong sense, 
then $\phi^{\Sc}(\bx):=\nabla\cdot\mathbf{F}^{\Sc}(\bx)\in 
H^1_{loc}(\mathbb{R}^3)$ satisfies the anisotropic Helmholtz scattering problem 
in a weak sense, with right-hand side given by the incoming field 
$\phi^{\In}=\nabla\cdot \mathbf{F^{\In}}$.
\end{lemma}
\begin{proof}
Letting $B_R$ be an open a ball centered at the origin that covers the support 
of $\epsilon(\bx)-I$, it is clear that the governing equation in  
the region $E = \mathbb{R}^3 \backslash B_R$ is simply the isotropic, 
homogeneous Helmholtz equation. Thus,
by standard results on the regularity of coefficients 
(Corollary 8.11, \cite{elliptic_pde}), the solution $\mathbf{F}_i^{\Sc}$ is 
infinitely differentiable in $E$.
We may, therefore, interpret the radiation condition in 
the strong sense. From the representation theorems 4.11 and 
4.13 in \cite{CK1} applied to the region $E$,
we find that $\phi^{\Sc}(\bx):=\nabla\cdot \bF^{\Sc}(\bx)$ satisfies the 
radiation condition (\ref{helmrad}).

Now let $\psi\in H^{1}(B_R)$. From eq. (\ref{Helmholtz_Vector}), we have
\begin{equation}
\begin{aligned}
\nabla\psi(\bx)\cdot\Big(\Delta \mathbf{F}^{\Sc}+\omega^2\mathbf{F}^{\Sc}+(\epsilon^{-1}-I)\nabla\nabla\cdot\mathbf{F}^{\Sc}\Big)=-\nabla\psi(\bx)\cdot(\epsilon^{-1}-I)\nabla\nabla\cdot\mathbf{F}^{\In}.
\end{aligned}
\end{equation}
Combined with the vector identity $\Delta \mathbf{F}^{\Sc}
=-\nabla\times\nabla\times\mathbf{F}^{\Sc}+\nabla\nabla\cdot\mathbf{F}^{\Sc}$, 
this yields
\begin{equation}
\nabla\psi\cdot\Big(-\nabla\times\nabla\times\mathbf{F}^{\Sc}+\omega^2\mathbf{F}^{\Sc}+\epsilon^{-1}\nabla\nabla\cdot\mathbf{F}^{\Sc}\Big)=-\nabla\psi(\bx)\cdot(\epsilon^{-1}-I)\nabla\nabla\cdot\mathbf{F}^{\In}.
\end{equation}
Integrating over the volume $B_R$ and using the divergence theorem, we obtain
\begin{equation}
\begin{aligned}
\int_{B_R}\nabla\psi(\bx)\cdot\epsilon^{-1}\nabla\nabla\cdot\mathbf{F}^{\Sc}-\omega^2\psi\nabla\cdot\mathbf{F}^{\Sc}dV&=\int_{\partial B_R}\psi\bn\cdot\Big(\nabla\times\nabla\times\mathbf{F}^{\Sc}-\omega^2\mathbf{F}^{\Sc}\Big)dV\\
-&\int_{B_R}\nabla\psi(\bx)\cdot(\epsilon^{-1}-I)\nabla\nabla\cdot\mathbf{F}^{\In}dV.
\end{aligned}
\end{equation}
This can be rewritten in the form
\begin{equation}
\begin{aligned}
&\int_{B_R}\nabla\psi(\bx)\cdot\epsilon^{-1}\nabla\nabla\cdot\mathbf{F}^{\Sc}-\omega^2\psi\nabla\cdot\mathbf{F}^{\Sc}dV\\
&=\int_{\partial B_R}\psi\bn\cdot\Big(-\Delta\mathbf{F}^{\Sc}-\omega^2\mathbf{F}^{\Sc}+\nabla\nabla\cdot\mathbf{F}^{\Sc}\Big)dV
-\int_{B_R}\nabla\psi(\bx)\cdot(\epsilon^{-1}-I)\nabla\nabla\cdot\mathbf{F}^{\In}dV.
\end{aligned}
\end{equation}
Since $\Delta\mathbf{F}^{\Sc}(\bx)+\omega^2\mathbf{F}^{\Sc}=0$ for $\bx\in \partial B_R$, we have the simpler equation:
\begin{equation}
\begin{aligned}
&\int_{B_R}\nabla\psi(\bx)\cdot\epsilon^{-1}\nabla\nabla\cdot\mathbf{F}^{\Sc}-\omega^2\psi\nabla\cdot\mathbf{F}^{\Sc}dV\\
&=\int_{\partial B_R}\psi T[\nabla\cdot \mathbf{F}^{\Sc}]dS
-\int_{B_R}\nabla\psi(\bx)\cdot(\epsilon^{-1}-I)\nabla\nabla\cdot\mathbf{F}^{\In}dV.
\end{aligned}
\end{equation}
It follows that $\phi^{\Sc}:=\nabla\cdot\mathbf{F}^{\Sc}$ satisfies  
(\ref{Helmholtz_Scalar_weak}) for $\phi^{\In}=\nabla\cdot \mathbf{F}^{\In}$,
the desired result.
\end{proof}

\begin{thm}[Uniqueness] The Vector Helmholtz scattering problem has at most 
one solution. \label{AVHP_unique}
\end{thm}
\begin{proof}
Let $\bF^{\Sc}\in H^2_{loc}(\mathbb{R}^3)$ be a solution of the homogeneous 
equation 
\begin{equation}\label{aug_helmholtz_homog}
\Delta \bF^{\Sc}+\omega^2\bF^{\Sc}+(\epsilon^{-1}-I)\nabla\nabla\cdot\bF^{\Sc}=0
\end{equation}
that satisfies the radiation condition.
From Lemma  \ref{Lemma1_div}, $\nabla\cdot\mathbf{F}^{\Sc}$ satisfies the 
homogeneous equation (\ref{Helmholtz_Scalar_weak}). 
Theorem \ref{Unique_Helmholtz_scalar} then shows that
$\nabla\cdot\mathbf{F}^{\Sc}=0$. Therefore,
we have that $\Delta\mathbf{F}^{\Sc}(\bx)+\omega^2\mathbf{F}^{\Sc}(\bx)=0$
for all $\bx \in \mathbb{R}^3$, so that $\mathbf{F}^{\Sc}=0$ for all 
$\bx \in \mathbb{R}^3$.
\end{proof}

In order to make use of the Fredholm alternative to complete our proof of
existence, we introduce the following operators:
\begin{equation}\label{Vw}
\begin{aligned}
\Vw(\bJ) &:=\int_{\mathbb R^3} \frac{e^{i\omega|\bx-\by|}}{4\pi|\bx-\by|}\bJ(\by)dV_{\by}, \\
\Tw(\bJ) &:=\frac{\bJ}{2}+\nabla\nabla\cdot \Vw(\bJ).
\end{aligned}
\end{equation}

We state the following lemma without proof 
(see \cite{CK2}, \cite{Potthast1}).

\begin{lemma}\label{lemma3}
The operator $\Tw-\To$ is compact on $L^2(\mathbb{R}^3)$.
\end{lemma}

\begin{thm}[Existence]\label{existence_Helmholtz}
The anisotropic scalar and vector Helmholtz scattering problems have solutions.
\end{thm}
\begin{proof}
Note first that
the vector field $\bF:=\Vw(\bJ)$ is a solution of 
eq. \eqref{Helmholtz_Vector} if and only if
\begin{equation}
-\bJ+(\epsilon^{-1}-I)\nabla\nabla\cdot \Vw(\bJ)=
-(\epsilon^{-1}-I)\nabla\nabla\cdot\bF^{\In}.
\end{equation}
Adding and substracting $\bJ/2$, this is equivalent to
\begin{equation}
-\bJ+(\epsilon^{-1}-I)\left(-\frac{\bJ}{2}+\Tw(\bJ)\right)=
-(\epsilon^{-1}-I)\nabla\nabla\cdot\bF^{\In}.
\end{equation}
Multiplying by $-2(\epsilon^{-1}+I)^{-1}$ we have
\begin{equation}\label{contractionK}
\bJ+2H_{\epsilon^{-1}}\To(\bJ)+2H_{\epsilon^{-1}}(\Tw-\To)(\bJ)=2H_{\epsilon^{-1}}\nabla\nabla\cdot\bF^{\In},
\end{equation}
where $H_{\epsilon^{-1}}$ is defined in \eqref{rhodef}.
By analogy with our earlier argument in the Laplace setting, 
we observe that the left-hand side of the resulting integral equation 
(\ref{contractionK}) is of 
the form $(I+B+K)\bJ$, where $I+B+K:L^2(B_R)\rightarrow L^2(B_R)$, with 
$B(\bJ) = 2\rho_{\epsilon}\To(\bJ)$,
$\|B\|<1$, and $K$ compact. Since $I+B$ is invertible, we can apply Fredholm 
theory directly.

Uniqueness for 
\eqref{contractionK} follows from Theorem \ref{AVHP_unique} and the
uniqueness of the representation $\bF=\Vw(\bJ)$.
It is shown in \cite{CK2} that
the operator $\Vw$ maps $L^2(B_R)\rightarrow H^2(B_R)$, 
so that $\mathbf{F}^{\Sc}:=\Vw(\bJ)\in H^2_{loc}(\mathbb{R}^3)$ 
for all $\bJ\in L^2(B_R)$. If, moreover, $\bJ$ satisfies \eqref{contractionK}, 
then by construction $\mathbf{F}^{\Sc}$ satisfies eq. (\ref{Helmholtz_Vector}), and $\nabla\cdot\mathbf{F}^{\Sc}\in H^1_{loc}(\mathbb{R}^3)$ satisfies eq. (\ref{Helmholtz_Scalar_weak}).
\end{proof}

To summarize: by solving the integral equation
\begin{equation}\label{eq:intH}
\Big( I+2H_{\epsilon^{-1}}\To+2H_{\epsilon^{-1}}(\Tw-\To) \Big) \bJ = 
2H_{\epsilon^{-1}}\nabla \phi^{\In},
\end{equation}
we obtain a solution to the vector Helmholtz scattering problem of the form
$\mathbf F =\Vw(\bJ)$. 
The function $\phi^{\Sc}:=\nabla\cdot \mathbf F$ 
provides a solution to the corresponding anisotropic Helmholtz scattering 
problem.  (The same result holds in 2D as well.)

\begin{remark}[Smoothness of the coefficients]
In this section we have assumed coefficients $\epsilon_{ij}(\bx)\in C^1(\mathbb{R}^3)$, instead of $L^{\infty}(\mathbb{R}^3)$ (as in the Laplace context)
 to be able to apply the unique continuation property. 
This regularity condition can be relaxed in various ways and the unique 
continuation property still holds. There is a vast literature on this subject
for second order elliptic partial differential equations 
(see \cite{ucp_4} for a good summary), following the early work of 
Carelman and M\"uller \cite{ucp_1,ucp_2}.
While it is known that $\epsilon_{ij}(\bx)\in L^{\infty}(\mathbb{R}^3)$
is too large a class of coefficients (due to counterexamples
\cite{ucp_0,ucp_00,ucp_000}), there has been a lot of effort at establishing
more general results \cite{ucp_3,ucp_5,ucp_6,ucp_7}.
The class of coefficients which are piecewise smooth where the jumps
occur on $C^2$ boundaries were studied in \cite{ucp_disc_1}.
In \cite{ucp_disc_2}, piecewise homogeneous objects were studied. 
It would be of great practical interest if
the unique continuation property holds for 
piecewise smooth coefficients, whose jumps occur on piecewise $C^2$ 
boundaries, allowing our integral formulation to be applicable to domains
with edges. This would follow naturally, since
the existence theorem (Theorem \ref{existence_Helmholtz}) 
only requires $\epsilon_{ij}(\bx)\in L^{\infty}(\mathbb{R}^3)$ and a 
uniqueness result for the anisotropic Helmholtz scattering problem. 
\end{remark}

\section{The anisotropic Maxwell's equations}
In this section we assume that $\epsilon(\bx)$ and $\mu(\bx)$ are real, 
symmetric $3\times 3$ matrices, uniformly positive definite with entries 
$\epsilon_{ij}(\bx), \mu_{ij}(\bx)\in C^2(\mathbb{R}^3)$. We also assume
that $\epsilon(\bx)-I$ and $\mu(\bx)-I$ have compact support, where I is 
the $3\times3$ identity matrix.

\begin{definition}
By the anisotropic Maxwell scattering problem, 
we mean the determination of functions 
$\bE^{\Sc},\bH^{\Sc}\in H_{loc}(curl,\mathbb{R}^3)$ (see \cite{Hodge_L2} 
for further details) such that:
\begin{equation}\label{Maxwell}
\begin{aligned}
\nabla\times\bE^{\Sc}(\bx)-i\omega\mu(\bx)\bH^{\Sc}(\bx)&=+i\omega(\mu(\bx)-I)\bH^{\In}(\bx),\\
\nabla\times\bH^{\Sc}(\bx)+i\omega\epsilon(\bx)\bE^{\Sc}(\bx)&=-i\omega(\epsilon(\bx)-I)\bE^{\In}(\bx),\\
\end{aligned}
\end{equation}
where the incoming field $\bE^{\In},\bH^{\In}$ satisfy the free space Maxwell's equations with $\epsilon(\bx)=I,\mu(\bx)=I$,
and the radiation condition:
\begin{equation}\label{EMrad}
\begin{array}{ll}
\bH^{\Sc}(\bx)\times\frac{\bx}{|\bx|}-\bE^{\Sc}(\bx)= \Big(\frac{1}{|\bx|}\Big)
& {\rm as}\ |\bx|\rightarrow \infty,
\end{array}
\end{equation}
uniformly in all directions $\frac{\bx}{|\bx|}$.
\end{definition}

\begin{thm}[Uniqueness]
The anisotropic Maxwell scattering problem has a unique solution.
\end{thm}

\begin{proof}
The proof is standard and based on the Rellich Lemma \cite{CK2}. 
Taking a ball $B_R$ that contains the support of $\epsilon(\bx)-I$ and $\mu(\bx)-I$, we have
\begin{equation}
\Re\Bigg(\int_{\partial B_R}\bn\times\bE^{\Sc}\cdot\overline{\bH}^{\Sc}dS\Bigg)=\Re\Bigg(i\omega\int_{B_R}\overline{\bH}^{\Sc T}\mu\bH^{\Sc}+\bE^{\Sc}\overline{\epsilon}\overline{\bE}^{\Sc}dV\Bigg)=0.
\end{equation}
Since $\bE^{\Sc}, \bH^{\Sc}$ are analytic in $\mathbb{R}^3/B_R$, 
we have that $\bE^{\Sc}(\bx)=\bH^{\Sc}(\bx)=0$ for all 
$\bx\in \mathbb{R}^3/B_R$. 
Using the unique continuation property \cite{Leis}, 
we obtain  $\bE^{\Sc}(\bx)=\bH^{\Sc}(\bx)=0$ for all $\bx\in \mathbb{R}^3$.
\end{proof}

In order to establish existence, we extend the technique described in
earlier sections.

\begin{thm}[Existence]
The anisotropic Maxwell scattering problem has a solution.
\end{thm}

\begin{proof}
We begin by rewriting the
anisotropic Maxwell scattering problem in a manner such that the 
variable coefficient terms only appear in the right hand side.
\begin{equation}\label{MaxwellCST}
\begin{aligned}
\nabla\times\bE^{\Sc}(\bx)-i\omega\bH^{\Sc}(\bx)&=+i\omega(\mu(\bx)-I)
\left(\bH^{\In}(\bx)+\bH^{\Sc}(\bx)\right),\\
\nabla\times\bH^{\Sc}(\bx)+i\omega\bE^{\Sc}(\bx)&=-i\omega(\epsilon(\bx)-I)
\left(\bE^{\In}(\bx)+\bE^{\Sc}(\bx)\right).\\
\end{aligned}
\end{equation}
We now define the right-hand sides as volume (polarization) currents:
\begin{equation}\label{eqs:JMdef}
\begin{aligned}
\bJ_V(\bx)&:=-i\omega(\epsilon(\bx)-I)\left(\bE^{\In}(\bx)+\bE^{\Sc}(\bx)\right),\\
\bM_V(\bx)&:=-i\omega(\mu(\bx)-I)\left(
\bH^{\In}(\bx)+\bH^{\Sc}(\bx)\right).\\
\end{aligned}
\end{equation}

Assume now that $\bE^{\Sc},\bH^{\Sc}\in C^1(\mathbb{R}^3)$ and that they
satisfy the constant coefficient Maxwell system:
\begin{equation}\label{PDEJM}
\begin{aligned}
\nabla\times\bE^{\Sc}(\bx)-i\omega\bH^{\Sc}(\bx)&=-\bM_V(\bx),\\
\nabla\times\bH^{\Sc}(\bx)+i\omega\bE^{\Sc}(\bx)&= \bJ_V(\bx).
\end{aligned}
\end{equation}
Then, applying the Stratton-Chu formulas (eq. 6.5 in \cite{CK2}), 
the following representation formula holds:
\begin{equation}\label{JM_formulation}
\begin{aligned}
\left( \begin{array}{cc}
\bE^{\Sc}\\
\bH^{\Sc} \end{array}\right)=
\left( \begin{array}{cc}
\frac{-1}{i\omega}(\nabla\nabla\cdot+\omega^2)\Vw& -\nabla\times \Vw\\
\nabla\times \Vw & \frac{-1}{i\omega}(\nabla\nabla\cdot+\omega^2)\Vw\end{array}\right)
\left( \begin{array}{cc}
\bJ_V \\
\bM_V \end{array}\right).
\end{aligned}
\end{equation}
Conversely, if $\bE^{\Sc}, \bH^{\Sc},\bJ_V,\bM_V$ satisfy 
\eqref{JM_formulation}, then it is a straightforward
computation to show that they satisfy \eqref{PDEJM}. 
Assuming that $\bJ_V,\bM_V$ are defined by \eqref{eqs:JMdef}, it is then
immediate to see that $\bE^{\Sc}, \bH^{\Sc}$ satisfy \eqref{MaxwellCST}, and 
equivalently \eqref{Maxwell}. Thus, we have proven that the PDE 
system with inhomogeneous coefficients \eqref{Maxwell} is equivalent to 
the integral formulation \eqref{JM_formulation}-\eqref{eqs:JMdef}.
It is also easy to show that the equivalence holds true for 
fields in $H_{loc}(curl,\mathbb{R}^3)$.

Eliminating $\bE^{\Sc}, \bH^{\Sc}$ from 
\eqref{eqs:JMdef}-\eqref{JM_formulation}, we obtain the following 
integral equation:
\begin{equation}\label{contraction_Maxwell}
\begin{aligned}
\bJ-H_{\epsilon}2T_{\omega}(\bJ)-2H_{\epsilon}\omega^2V_{\omega}(\bJ)-i\omega2H_{\epsilon}\nabla\times V_{\omega}(\bM) &=2H_{\epsilon}\bE^{\In},\\
\bM-H_{\mu}2T_{\omega}(\bM)-2H_{\mu}\omega^2V_{\omega}(\bM)+i\omega2H_{\mu}\nabla\times V_{\omega}(\bJ) &=2H_{\mu}\bH^{\In},\\
\end{aligned}
\end{equation}
where $H_{\epsilon}$ is defined in \eqref{rhodef} and
$H_{\mu}(\bx):=(\mu(\bx)+I)^{-1}(\mu(\bx)-I)$. We also make the change of
variables
\begin{equation}
\bJ:=\frac{\bJ_V}{-i\omega},\qquad
\bM:=\frac{\bM_V}{-i\omega},
\end{equation}
to avoid low frequency breakdown (that is, instability of the representation
as $\omega \rightarrow 0$).

Using an approach similar to that in previous sections
(in the function space $L^2(B_R)\times L^2(B_R)$),
we can write eq. \eqref{contraction_Maxwell} in the form 
\[ \big(I+C+K \big) 
\left( \begin{array}{cc}
\bJ \\
\bM \end{array}\right)
=
\left( \begin{array}{cc}
2\rho_{\epsilon}\bE^{\In} \\
2\rho_{\mu}\bH^{\In}
\end{array}\right) \, ,
\] 
where $C$ is a contraction and $K$ is compact. 
We now use the standard representation for electromagnetic fields
in terms of electric and magnetic currents:
\begin{equation}\label{represent_Maxwell}
\begin{aligned}
\bE^{\Sc}&=\nabla\nabla\cdot V_{\omega}(\bJ)+\omega^2V_{\omega}(\bJ)+i\omega\nabla\times V_{\omega}(\bM),\\
\bH^{\Sc}&=\nabla\nabla\cdot V_{\omega}(\bM)+\omega^2V_{\omega}(\bM)-i\omega\nabla\times V_{\omega}(\bJ).
\end{aligned}
\end{equation}
Since the operators involved in (\ref{represent_Maxwell}) 
map $L^2(B_R)$ into $H_{loc}(curl, \mathbb{R}^3)$, uniqueness of the 
anisotropic Maxwell scattering problem implies uniqueness, and hence, 
existence for the integral equation \eqref{contraction_Maxwell}.
Finally, this yields existence for the anisotropic Maxwell scattering problem
itself.
\end{proof}

To summarize, solving the integral system
\begin{equation}\label{eq:intM}
\begin{aligned}
\bJ-\rho_{\epsilon}2T_{\omega}(\bJ)-2\rho_{\epsilon}\omega^2V_{\omega}(\bJ)-i\omega2\rho_{\epsilon}\nabla\times V_{\omega}(\bM) &=2\rho_{\epsilon}\bE^{\In},\\
\bM-\rho_{\mu}2T_{\omega}(\bM)-2\rho_{\mu}\omega^2V_{\omega}(\bM)+i\omega2\rho_{\mu}\nabla\times V_{\omega}(\bJ) &=2\rho_{\mu}\bH^{\In},\\
\end{aligned}
\end{equation}
and computing 
\begin{equation}
\begin{aligned}
\bE^{\Sc}&=\nabla\nabla\cdot V_{\omega}(\bJ)+\omega^2V_{\omega}(\bJ)+i\omega\nabla\times V_{\omega}(\bM),\\
\bH^{\Sc}&=\nabla\nabla\cdot V_{\omega}(\bM)+\omega^2V_{\omega}(\bM)-i\omega\nabla\times V_{\omega}(\bJ),
\end{aligned}
\end{equation}
provides a solution to the anisotropic Maxwell scattering problem.

For the particular case $\mu=I$, we have the following, simpler 
integral equation:
\begin{equation}\label{contraction_Maxwell_reduced}
\begin{aligned}
\bJ-\rho_{\epsilon}2T_{\omega}(\bJ)-2\rho_{\epsilon}\omega^2V_{\omega}(\bJ) &=2\rho_{\epsilon}\bE^{\In},
\end{aligned}
\end{equation}
and the corresponding representation
\begin{equation}\label{represent_Maxwell_reduced}
\begin{aligned}
\bE^{\Sc}&=\nabla\nabla\cdot V_{\omega}(\bJ)+\omega^2V_{\omega}(\bJ), \\
\bH^{\Sc}&=-i\omega\nabla\times V_{\omega}(\bJ).
\end{aligned}
\end{equation}

A closely related integral formulation (using a slightly different scaling)
is widely used \cite{Volume_IEEE_1,Volume_IEEE_2,Volume_IEEE_3}, 
and known as the ``JM" volume integral formulation.

\begin{remark}[Non-smooth coefficients and lossy materials]
The unique continuation property for the Maxwell system 
(\ref{Maxwell}) has been extended to the case 
$\epsilon_{ij}(\bx),\mu_{ij}(\bx)\in C^1(\mathbb{R}^3)$ in 
\cite{ucp_Maxwell_C1} and to the case of Lipschitz coefficients 
$\epsilon_{ij}(\bx),\mu_{ij}(\bx)\in W^{1,\infty}(B_R)$ in 
\cite{ucp_Maxwell_Lip_1,ucp_Maxwell_Lip_2}. In both settings, 
$\epsilon(\bx)$ and $\mu(\bx)$ are assumed to be real (no dissipation). 

Lossy materials for which
the unique continuation property has been shown to hold 
\cite{Potthast2}
includes the case when $\mu(\bx)=I$ and $\epsilon(\bx)$ has entries 
$\epsilon_{ij}(\bx)\in C^3(\mathbb{R}^3)$ 
with $\epsilon(\bx)=U_1(\bx)D_{\epsilon}(\bx)U_1^*(\bx)$, 
where $U_1(\bx)$ is a unitary complex matrix
and $D_{\epsilon}(\bx)$ is diagonal with diagonal entries 
whose real parts are positive and bounded away from zero 
and whose imaginary parts are non-negative. 

Note that in the proof of existence described in the previous theorem, 
$\epsilon(\bx)$ and $\mu(\bx)$ are assumed to be real symmetric, 
with entries in $C^2(\mathbb{R}^3)$. 
Assuming the unique continuation property holds,
extension to the complex dissipative case where both matrices 
$\epsilon(\bx)$ and $\mu(\bx)$ have $L^{\infty}$ entries is straightforward. 
By this, we mean that 
$\epsilon(\bx)=U_1(\bx)D_{\epsilon}(\bx)U_1^*(\bx)$, 
$\mu(\bx)=U_2(\bx)D_{\mu}(\bx)U_2^*(\bx)$, 
where $D_{\epsilon}$ and $D_{\mu}$ have diagonal entries with strictly positive real part and non-negative 
imaginary part.
For further discussion, see 
\cite{ucp_Maxwell_general,ucp_Maxwell_general_2,ucp_Maxwell_general_3}.
\end{remark}

\section{Numerical results}

We illustrate the performance of our approach by 
solving the integral equations \eqref{eq:intH} and 
\eqref{contraction_Maxwell_reduced}. 
We begin with a uniform $n \times n \times n$ 
mesh on which we discretize the incoming field, the
material properties, and the unknown solution vectors $\bJ$ and/or $\bM$.
We apply the various integral operators that arise
using Fourier methods, as described in \cite{Volintfft}.
Very briefly, the method proceeds
by (a) truncating the governing free-space Green's function (limited
to the user-specified range over which we seek the solution),
(b) transforming the truncated kernel - yielding a
smooth function in Fourier space, and (c) imposing a high frequency 
cutoff defined by the grid spacing of the resolving mesh. Assuming 
that the data is well-resolved on this mesh, the method achieves high-order 
(superalgebraic) convergence.
The linear systems are solved iteratively, using Bi-CGStab \cite{bicgstab}.

For the sake of simplicity, 
we let $\mu(\bx)=I$ and study the influence of $\epsilon$ on the 
behavior of the numerical method. There are three parameters to 
consider.  First is the {\em contrast}, defined as the maximum ratio 
between the eigenvalues of $\epsilon$ and the background dielectric constant.
Second is the level of anisotropy, determined by the ration of the
eigenvalues of $\epsilon$ (as well as rotations of $\epsilon$ to 
nondiagonal form).
 
We assume that the computational domains is set to $[0,1]^3$.
We define a bump function 
\[
W(x,y,z) = e^{-\big(\frac{x-0.5}{0.25}\big)^8}e^{-\big(\frac{y-0.5}{0.25}\big)^8}e^{-\big(\frac{z-0.5}{0.25}\big)^8}.
\]
which has decayed to zero at the edge of the computational domain to machine precision.

\subsection{Isotropic scattering from a highly oscillatory structure}

In our first example, we consider the interaction of an
electromagnetic wave $\bE^{\In} =(0,0,\exp (i\omega x))$ with 
a highly oscillatory but
locally isotropic permittivity:
 \begin{equation}\label{1stEx}
  \epsilon(x,y,z)=\Big(
1+W(x,y,z)\big(1+.1\sin(\omega x)\sin(\omega y)\sin(\omega z)\big)
\Big)I,
  \end{equation}
where $I$ is the $3\times3$ identity matrix and $\omega=408$.
Note that the contrast is approximately 2 and that 
the magnitude of the oscillation is relatively small: 10\% of the 
magnitude of the bump function $W(x,y,z)$ itself. Nevertheless,
to resolve $\epsilon$ at 2 points per wavelength
requires at least 200 points in each component direction.
We plot the $z$ component of the scattered field $\bE^{\Sc}$ in 
Figure \ref{FigScat} after solving the integral equation 
\eqref{contraction_Maxwell_reduced}. 

\begin{figure}[htb]
\begin{center}
\includegraphics[width=4in,trim={3.8cm 7.7cm 3.5cm 7.7cm},clip]{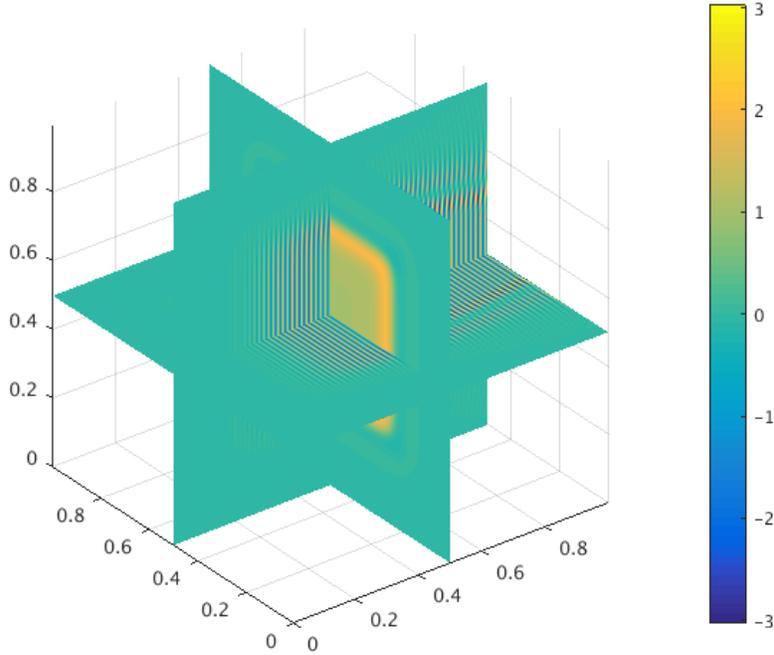}%
\end{center}
\caption{The $z$-component of the electric field,
when solving the Maxwell scattering problem \eqref {Maxwell} with $\omega=408$, $\mu=I$ and $\epsilon$ defined in \eqref{1stEx}.
The solution is obtained 
from the representation \eqref{represent_Maxwell_reduced} and the 
corresponding integral equation 
\eqref{contraction_Maxwell_reduced}, discretized with 
$650^3$ points.}
\label{FigScat}
\end{figure}

Since we do not have an exact solution for this problem, we carry out a 
numerical convergence study, using a $326 \times 326 \times 326$ grid followed
by a $650 \times 650 \times 650$ grid, suggesting that 
six digits of accuracy have been achieved
on the coarser grid in both the $L_2$ and $L_\infty$ norms.
{ The calculation required 61 matrix-vector multiplies and 152 minutes on an Intel Xeon 2.5GHz workstation with 60 cores and 1.5 terabytes of memory.}

 \subsection{Strong isotropic and anisotropic scattering}

To study the behavior of our integral equation formulation 
at higher contrast over a range
of frequencies, we consider two additional locally isotropic examples
and two anisotropic ones.
For the isotropic cases, we let
\[
 \epsilon_{222}(x,y,z) := \big(1+W(x,y,z)\big) I, \qquad
 \epsilon_{444}(x,y,z) = \big(1+3W(x,y,z)\big) I.
\] 
Note that $\epsilon_{222}$ has a maximum contrast of 2, while
$\epsilon_{444}$ has a maximum contrast of 4.
For the anisotropic examples, we let 
\begin{equation}\label{eps}
\epsilon_{234}(x,y,z)=\left( \begin{array}{ccc}
\epsilon_{1} & 0 & 0\\
0 & \epsilon_{2} & 0\\
0 & 0 & \epsilon_{3}\\
\end{array} \right),
\text{ with }
\left\{
\begin{array}{l}
\epsilon_{1}(x,y,z)=1+W(x,y,z),\\
\epsilon_{2}(x,y,z)=1+2W(x,y,z),\\
\epsilon_{3}(x,y,z)=1+3W(x,y,z),
\end{array}\right.
\end{equation} 
and
\begin{equation} \label{eps_dense}
\epsilon_{dense}(x,y,z)=R_2(\phi)R_1(\theta)\left( \begin{array}{ccc}
\epsilon_{1} & 0 & 0\\
0 & \epsilon_{2} & 0\\
0 & 0 & \epsilon_{3}\\
\end{array} \right)R_1(\theta)'R_2(\phi)'
\end{equation}
where $\epsilon_{1}$, $\epsilon_{2}$, $\epsilon_{3}$ 
are defined in \eqref{eps}, and the rotation matrices 
$R_1$ and $R_2$ are given by 
\begin{equation}
\begin{aligned}
R_1(\theta)&=\left( \begin{array}{ccc}
1&0&0\\
0&\cos(\theta) & -\sin(\theta)\\
0&\sin(\theta) & \cos(\theta)
\end{array} \right), \quad
R_2(\phi)&=\left( \begin{array}{ccc}
\cos(\phi) & -\sin(\phi) & 0 \\
\sin(\phi) & \cos(\phi) & 0\\
0&0&1
\end{array} \right),
\end{aligned}
\end{equation}
with $\phi(x,y,z)=\pi x$ and $\theta(x,y,z)=\pi y$.

We first examine
the performance of Bi-CGStab, plotting the 
number of iterations required to achieve a tolerance of $10^{-14}$ 
as a function of the frequency $\omega$ for both the Helmholtz and 
Maxwell scattering problems on a $150 \times 150 \times 150$ grid
(Fig. \ref{FigComp}).
As expected, the number of iterations increases with frequency. 
Moreover, for a fixed frequency, the number of iterations increases 
with the contrast. Note, however, that the anisotropy and rotation 
have only limited impact on the number of iterations. 
Clearly, while the method is robust at low frequencies, these calculations
remain challenging in strong scattering regimes. 

\begin{figure}[H]
\begin{center}
\includegraphics[width=6in]{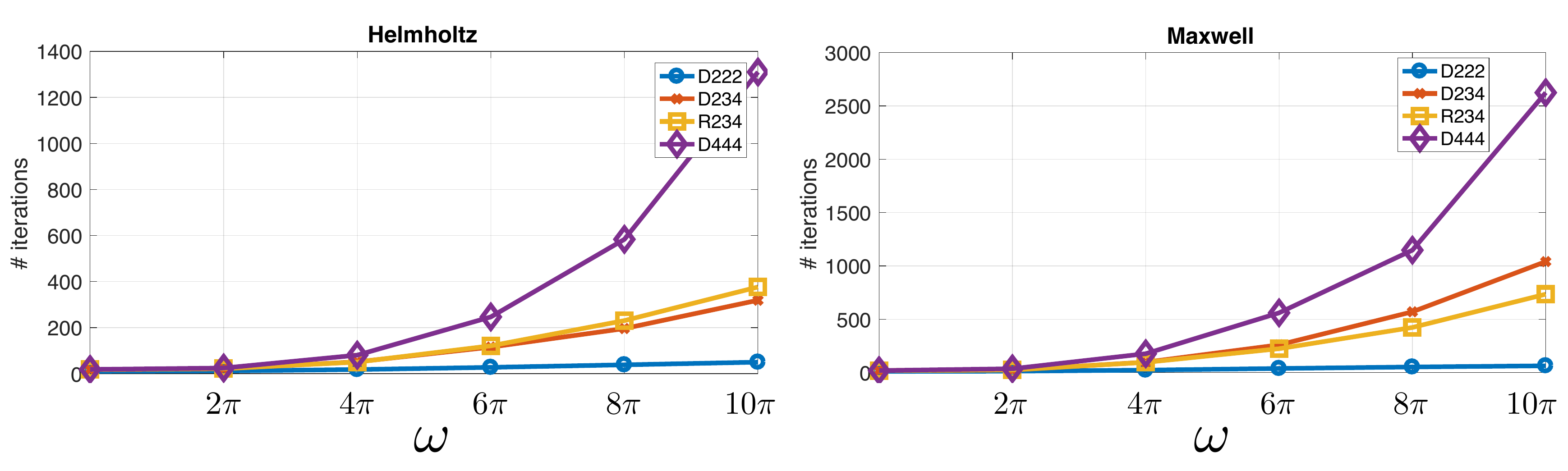}%
\end{center}
\caption{Number of iterations for convergence of  Bi-CGStab with a
tolerance of $10^{-14}$, as a function of the frequency $\omega$ when
applied to the 
Helmholtz integral equation \eqref{eq:intH} (left) and the 
Maxwell integral equation \eqref{contraction_Maxwell_reduced} (right).
The labels D222, D234, R234 and D444 correspond to 
$\epsilon_{222}$, $\epsilon_{234}$, $\epsilon_{dense}$, and $\epsilon_{444}$,
respectively.
}
\label{FigComp}
\end{figure}

Additional details regarding numerical experiments are provided in 
Tables \ref{table2}-\ref{table5}.
Note that the number of iterations is more or less constant
for each fixed problem as the mesh is refined (consistent with the expected
behavior of a second kind Fredholm equation).
Note also that the number of iterations for the 
diagonally anisotropic case $\epsilon_{234}$ (Table \ref{table3}) is about the
same as for the case $\epsilon_{dense}$, 
where the principal axes have been rotated throughout the domain 
(Table \ref{table4}).

In the example $\epsilon_{444}$, 
the number of iterations is significantly
worse than in any of the other cases, even though it is 
locally isotropic. Thus, the behavior of the 
proposed integral equation (\ref{contraction_Maxwell_reduced}) appears 
to be controlled by the contrast and frequency more than by 
anisotropy. The scatterer is approximately
$\left(\frac{\omega}{2 \pi} \, \sqrt{4} \right)^3$ cubic wavelengths in
size - on the order of 1000 for the largest value of $\omega$.
Thus, it is not surprising that Bi-CGStab requires
many iterations to converge.

\begin{table}[H]
\caption{Summary of numerical results for the locally 
isotropic permittivity tensor $\epsilon_{222}$.
The size is defined to be the number of wavelengths of the 
incoming field across the unit box supporting the perturbation $\epsilon$,
namely $\frac{\omega}{2 \pi}$.
$N_{tot}$ denotes the total number of unknowns, $n_{side}$ denotes
the number of points in each linear dimension, $E_2$ denotes the $L_2$ error,
$E_\infty$ denotes the $L_\infty$ error, $N_{matvec}$ denotes the 
number of matrix-vector multiplies required, and $Time$ 
denotes the total time required for the computation in seconds.}
\centering
\begin{tabular}{crccccc}  \\ [-0.3cm]
\hline
Size $\left( \frac{\omega}{2 \pi} \right)$ 
& $N_{tot}$ & $n_{side}$ & $E_2$ & $E_{\infty}$ & 
$N_{matvec}$ & $Time$ $(s)$ \\
\hline \\ [-0.3cm]
$10^{-50}$ & $1.03 \, 10^6$ & $70$ & $1.5\times10^{-7}$ & $1.1\times10^{-6}$ & $22$ & $6.4$ \\ [0.1cm]
$10^{-50}$ & $5.18 \, 10^6$ & $120$ & $1.2\times10^{-11}$ & $7.8\times10^{-11}$ & $22$ & $27.7$ \\ [0.1cm]
$10^{-50}$ & $1.75 \, 10^7$ & $180$ & $1.6\times10^{-16}$ & $1.0\times10^{-15}$ & $22$ & $93.6$ \\ [0.1cm]
\hline
$10^{-10}$ & $1.03 \, 10^6$ & $70$ & $1.5\times10^{-7}$ & $1.1\times10^{-6}$ & $22$ & $6.6$ \\ [0.1cm]
$10^{-10}$ & $5.18 \, 10^6$ & $120$ & $1.2\times10^{-11}$ & $7.8\times10^{-11}$ & $22$ & $27.5$ \\ [0.1cm]
$10^{-10}$ & $1.75 \, 10^7$ & $180$ & $1.6\times10^{-16}$ & $1.0\times10^{-15}$ & $22$ & $91.9$ \\ [0.1cm]
\hline
$1$ & $1.03 \, 10^6$ & $70$ & $1.5\times10^{-7}$ & $8.2\times10^{-7}$ & $27$ & $8.1$ \\ [0.1cm]
$1$ & $5.18 \, 10^6$ & $120$ & $1.1\times10^{-11}$ & $6.0\times10^{-11}$ & $27$ & $33.3$ \\ [0.1cm]
$1$ & $1.75 \, 10^7$ & $180$ & $2.4\times10^{-16}$ & $8.6\times10^{-16}$ & $27$ & $110.1$ \\ [0.1cm]
\hline
$20$ & $1.03 \, 10^6$ & $70$ & $7.3\times10^{-4}$ & $1.1\times10^{-3}$ & $469$ & $134.1$ \\ [0.1cm]
$20$ & $5.18 \, 10^6$ & $120$ & $4.9\times10^{-10}$ & $1.2\times10^{-9}$ & $426$ & $487.5$ \\ [0.1cm]
$20$ & $1.75 \, 10^7$ & $180$ & $1.9\times10^{-14}$ & $1.8\times10^{-14}$ & $426$ & $1614.8$ \\ [0.1cm]
\hline
\end{tabular}
\label{table2}
\end{table}

\begin{table}[H]
\caption{Summary of numerical results for the anisotropic
permittivity tensor $\epsilon_{234}$. See 
Table \ref{table2} for an explanation of the column headers.}
\centering
\begin{tabular}{crccccc}  \\ [-0.3cm]
\hline
Size $\left( \frac{\omega}{2 \pi} \right)$ 
& $N_{tot}$ & $n_{side}$ & $E_2$ & $E_{\infty}$ & $N_{matvec}$ & $Time$ $(s)$ \\ 
\hline \\ [-0.3cm]
$10^{-50}$ & $1.03 \, 10^6$ & $70$ & $3.1\times10^{-6}$ & $1.9\times10^{-5}$ & $37$ & $10.9$ \\ [0.1cm]
$10^{-50}$ & $5.18 \, 10^6$ & $120$ & $1.7\times10^{-9}$ & $8.5\times10^{-9}$ & $37$ & $44.9$ \\ [0.1cm]
$10^{-50}$ & $1.75 \, 10^7$ & $180$ & $1.7\times10^{-13}$ & $9.9\times10^{-13}$ & $37$ & $148.8$ \\ [0.1cm]
\hline
$10^{-10}$ & $1.03 \, 10^6$ & $70$ & $3.1\times10^{-6}$ & $1.9\times10^{-5}$ & $37$ & $10.8$ \\ [0.1cm]
$10^{-10}$ & $5.18 \, 10^6$ & $120$ & $1.7\times10^{-9}$ & $8.5\times10^{-9}$ & $37$ & $45.4$ \\ [0.1cm]
$10^{-10}$ & $1.75 \, 10^7$ & $180$ & $1.7\times10^{-13}$ & $9.9\times10^{-13}$ & $37$ & $148.5$ \\ [0.1cm]
\hline
$1$ & $1.03 \, 10^6$ & $70$ & $2.2\times10^{-6}$ & $1.0\times10^{-5}$ & $48$ & $14.4$ \\ [0.1cm]
$1$ & $5.18 \, 10^6$ & $120$ & $1.2\times10^{-9}$ & $4.3\times10^{-9}$ & $48$ & $57.2$ \\ [0.1cm]
$1$ & $1.75 \, 10^7$ & $180$ & $1.2\times10^{-13}$ & $6.4\times10^{-13}$ & $50$ & $198.1$ \\ [0.1cm]
\hline
$5$ & $1.03 \, 10^6$ & $70$ & $1.8\times10^{-6}$ & $7.7\times10^{-6}$ & $2125$ & $611.8$ \\ [0.1cm]
$5$ & $5.18 \, 10^6$ & $120$ & $9.6\times10^{-10}$ & $3.0\times10^{-9}$ & $2081$ & $2378.1$ \\ [0.1cm]
$5$ & $1.75 \, 10^7$ & $180$ & $8.9\times10^{-14}$ & $3.6\times10^{-13}$ & $2105$ & $7948.1$ \\ [0.1cm]
\hline
\end{tabular}
\label{table3}
\end{table}

\begin{table}[H]
\caption{Summary of numerical results for the anisotropic
permittivity tensor $\epsilon_{dense}$. See 
Table \ref{table2} for an explanation of the column headers.}
\centering
\begin{tabular}{crccccc}  \\ [-0.3cm]
\hline
Size $\left( \frac{\omega}{2 \pi} \right)$ 
& $N_{tot}$ & $n_{side}$ & $E_2$ & $E_{\infty}$ & $N_{matvec}$ & $Time$ $(s)$ \\ 
\hline \\ [-0.3cm]
$10^{-50}$ & $1.03 \, 10^6$ & $70$ & $1.3\times10^{-6}$ & $7.7\times10^{-6}$ & $37$ & $10.6$ \\ [0.1cm]
$10^{-50}$ & $5.18 \, 10^6$ & $120$ & $4.1\times10^{-10}$ & $1.8\times10^{-9}$ & $37$ & $45.3$ \\ [0.1cm]
$10^{-50}$ & $1.75 \, 10^7$ & $180$ & $2.0\times10^{-14}$ & $1.6\times10^{-13}$ & $37$ & $149.2$ \\ [0.1cm]
\hline
$10^{-10}$ & $1.03 \, 10^6$ & $70$ & $1.3\times10^{-6}$ & $7.7\times10^{-6}$ & $37$ & $10.8$ \\ [0.1cm]
$10^{-10}$ & $5.18 \, 10^6$ & $120$ & $4.1\times10^{-10}$ & $1.8\times10^{-9}$ & $37$ & $45.2$ \\ [0.1cm]
$10^{-10}$ & $1.75 \, 10^7$ & $180$ & $2.0\times10^{-14}$ & $1.6\times10^{-13}$ & $37$ & $147.9$ \\ [0.1cm]
\hline
$1$ & $1.03 \, 10^6$ & $70$ & $1.3\times10^{-6}$ & $6.9\times10^{-6}$ & $50$ & $14.7$ \\ [0.1cm]
$1$ & $5.18 \, 10^6$ & $120$ & $4.1\times10^{-10}$ & $2.6\times10^{-9}$ & $51$ & $61.0$ \\ [0.1cm]
$1$ & $1.75 \, 10^7$ & $180$ & $1.9\times10^{-14}$ & $1.1\times10^{-13}$ & $51$ & $202.0$ \\ [0.1cm]
\hline
$5$ & $1029000$ & $70$ & $1.7\times10^{-6}$ & $1.1\times10^{-5}$ & $1447$ & $412.6$ \\ [0.1cm]
$5$ & $5.18 \, 10^6$ & $120$ & $5.1\times10^{-10}$ & $2.9\times10^{-9}$ & $1474$ & $1697.6$ \\ [0.1cm]
$5$ & $1.75 \, 10^7$ & $180$ & $1.3\times10^{-13}$ & $2.3\times10^{-13}$ & $1478$ & $5687.2$ \\ [0.1cm]
\hline
\end{tabular}
\label{table4}
\end{table}

\begin{table}[H]
\caption{Summary of numerical results for the locally isotropic
permittivity tensor $\epsilon_{444}$. See 
Table \ref{table2} for an explanation of the column headers.}
\centering
\begin{tabular}{crccccc}  \\ [-0.3cm]
\hline
Size $\left( \frac{\omega}{2 \pi} \right)$ 
& $N_{tot}$ & $n_{side}$ & $E_2$ & $E_{\infty}$ & $N_{matvec}$ & $Time$ $(s)$ \\ 
\hline \\ [-0.3cm]
$10^{-50}$ & $3.75 \, 10^5$ & $50$ & $8.4\times10^{-5}$ & $4.5\times10^{-4}$ & $37$ & $5.0$ \\ [0.1cm]
$10^{-50}$ & $3.00 \, 10^6$ & $100$ & $3.8\times10^{-8}$ & $2.1\times10^{-7}$ & $38$ & $27.9$ \\ [0.1cm]
$10^{-50}$ & $1.01 \, 10^7$ & $150$ & $1.5\times10^{-11}$ & $9.3\times10^{-11}$ & $38$ & $87.7$ \\ [0.1cm]
$10^{-50}$ & $3.19 \, 10^7$ & $220$ & $8.5\times10^{-15}$ & $1.5\times10^{-14}$ & $38$ & $271.9$ \\ [0.1cm]
\hline
$10^{-10}$ & $3.75 \, 10^5$ & $50$ & $8.4\times10^{-5}$ & $4.5\times10^{-4}$ & $37$ & $5.1$ \\ [0.1cm]
$10^{-10}$ & $3.00 \, 10^6$ & $100$ & $3.8\times10^{-8}$ & $2.1\times10^{-7}$ & $38$ & $27.9$ \\ [0.1cm]
$10^{-10}$ & $1.01 \, 10^7$ & $150$ & $1.5\times10^{-11}$ & $9.3\times10^{-11}$ & $38$ & $88.7$ \\ [0.1cm]
$10^{-10}$ & $3.19 \, 10^7$ & $220$ & $8.6\times10^{-15}$ & $1.5\times10^{-14}$ & $38$ & $273.9$ \\ [0.1cm]
\hline
$1$ & $3.75 \, 10^5$ & $50$ & $6.9\times10^{-5}$ & $2.5\times10^{-4}$ & $61$ & $7.9$ \\ [0.1cm]
$1$ & $3.00 \, 10^6$ & $100$ & $3.1\times10^{-8}$ & $1.3\times10^{-7}$ & $69$ & $49.5$ \\ [0.1cm]
$1$ & $1.01 \, 10^7$ & $150$ & $1.2\times10^{-11}$ & $4.5\times10^{-11}$ & $72$ & $161.5$ \\ [0.1cm]
$1$ & $3.19 \, 10^7$ & $220$ & $1.6\times10^{-13}$ & $1.7\times10^{-13}$ & $68$ & $472.3$ \\ [0.1cm]
\hline
$5$ & $3.75 \, 10^5$ & $50$ & $6.7\times10^{-5}$ & $2.4\times10^{-4}$ & $5774$ & $704.0$ \\ [0.1cm]
$5$ & $3.00 \, 10^6$ & $100$ & $2.9\times10^{-8}$ & $1.3\times10^{-7}$ & $5418$ & $3768.8$ \\ [0.1cm]
$5$ & $1.01 \, 10^7$ & $150$ & $1.2\times10^{-11}$ & $4.4\times10^{-11}$ & $5251$ & $11489.8$ \\ [0.1cm]
$5$ & $3.19 \, 10^7$ & $220$ & $2.3\times10^{-12}$ & $3.1\times10^{-12}$ & $5410$ & $36594.8$ \\ [0.1cm]
\hline
\end{tabular}
\label{table5}
\end{table}

\begin{figure}[H]
\begin{center}
\includegraphics[width=4in]{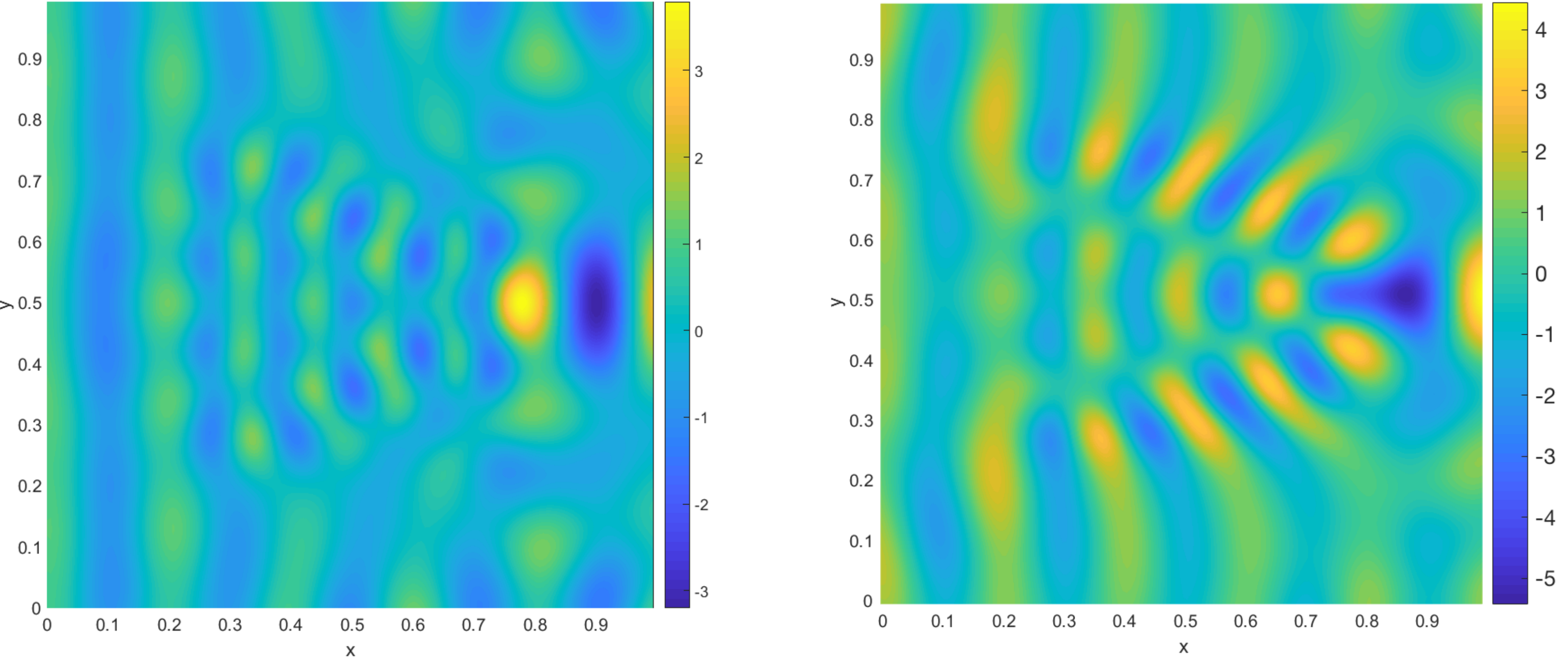} 
\end{center}
\caption{A slice of the $z$-component of the electric field at $z=1/2$ (left)
when solving the Maxwell scattering problem with $\omega=10 \pi$, 
$\mu=I$ and $\epsilon = \epsilon_{dense}$ defined in \eqref{eps_dense}
using a $150\times 150\times 150$ grid.
A slice of the real part of the acoustic field at $z=1/2$ (right) 
when solving the
Helmholtz scattering problem with $\omega=10 \pi$, 
and $\epsilon = \epsilon_{234}$ defined in \eqref{eps}
using a $150\times 150\times 150$ grid.
}
\end{figure}

\section{Discussion}

We have presented a collection of Fredholm integral equations
for electrostatic, acoustic and electromagnetic scattering 
problems in anisotropic, inhomogeneous media.
In the electrostatic and acoustic cases, our approach appears to be new,
and involves recasting the scalar problem of interest in terms of
a vector unknown. We have shown that high order accuracy can be 
achieved using the truncated kernel method of 
\cite{Volintfft} and the FFT. We have also shown
that problems with low or moderate contrast 
are rapidly solved using the Bi-CGStab iterative method,
even with nearly one billion unknowns on a single multicore workstation.
Once the domain is several wavelengths in size, however,
and the contrast is large, we have found that both Bi-CGStab and GMRES
perform rather poorly.
In our largest high-contrast example, 
the scatterer is approximately 1000 cubic wavelengths in size, and 
iterative methods would be expected to require 
many iterations to converge.
This suggests two avenues for further research:
either the development of fast, direct solvers 
(the truncated kernel method of \cite{Volintfft} can provide explicit
matrix entries) or a preconditioning strategy suitable for this class
of problems. We are currently investigating both lines of research and
will report our progress at a later date.

\bibliographystyle{unsrt}
\bibliography{referencias}

\end{document}